\documentclass[12pt]{amsproc}
\usepackage{amssymb}
\usepackage{amsmath}
\usepackage{amsfonts}
\usepackage{geometry}
\usepackage{graphicx}
\providecommand{\U}[1]{\protect\rule{.1in}{.1in}}
\theoremstyle{plain}

\newtheorem{corollary}{Corollary}

\newtheorem{lemma}{Lemma}

\newtheorem{theorem}{Theorem}
\numberwithin{equation}{section}
\newcommand{\lap}{\mbox{$\triangle$}}

\begin{document}
\title[Quantitative uniqueness of elliptic equations]{Quantitative uniqueness of elliptic equations}
\author{Jiuyi Zhu }
\address{Department of Mathematics\\
Johns Hopkins University\\
Baltimore, MD 21218, USA\\
Emails:  jzhu43@math.jhu.edu }
\thanks{\noindent }
\date{}
 \dedicatory{ }

\begin{abstract}
Based on a variant of frequency function, we improve the vanishing
order of solutions for Schr\"{o}dinger equations which describes
quantitative behavior of strong uniqueness continuation property.
For the first time, we investigate the quantitative uniqueness of
higher order elliptic equations and show the vanishing order of
solutions. Furthermore, strong unique continuation is established
for higher order elliptic equations using this variant of frequency
function.

\end{abstract}

\maketitle
\section{Introduction}

We say the vanishing order of solution at $x_0$ is $l$, if $l$ is
the largest integer such that $D^{\alpha}u(x_0)=0$ for all
$|\alpha|\leq l$. It describes quantitative behavior of strong
unique continuation property.  It is well known that all zeros of
nontrivial solutions of second order  linear equations on smooth
compact Riemannian manifolds are of finite order \cite{Ar}. In the
papers \cite{DF} and \cite{DF1}, Donnelly and Fefferman showed that
if $u$ is an eigenfunction on a compact smooth Riemannian manifold
$\mathcal{M}$, that is,
 $$-\lap_{\mathcal{M}} u=\lambda u, \quad \quad \mbox{in} \ \mathcal{M}$$
for some $\lambda>0$, then the maximal vanishing order of $u$ on
$\mathcal{M}$ is less than $C\sqrt{\lambda}$, here $C$ only depends
on the manifold $\mathcal{M}$.
 Kukavica
in \cite{Ku} considered the vanishing order of solutions of
Schr\"{o}dinger equation
\begin{equation}
-\lap_{\mathcal{M}} u=V(x)u, \label{schro}
\end{equation}
where $V(x)\in L^\infty(\mathcal{M})$. He established that the
vanishing order of solution in (\ref{schro}) is everywhere less than
\begin{equation}
C(1+(\sup_{\Omega}V_+)^{\frac{1}{2}}+osc(V)^2),\label{order}
\end{equation}
 where $V_+(x)=\max\{V(x), \ 0\}$, $osc(V)=\sup V-\inf V$
and $C$ only depends on the underlying domain $\mathcal{M}$. If
$V\in C^1(\mathcal{M})$, Kukavica was able to show that the upper
bound of vanishing order is less than $C(1+\|V\|_{C^1})$,
 where $\|V\|_{C^1}=\|V\|_{L^\infty}+\|\nabla
V\|_{L^\infty}$. Based on the Donnelly and Fefferman's  work in
\cite{DF}, Kukavica conjectured that the rate of of vanishing order
of $u$ is less than $C(1+\|V\|_{L^\infty}^{\frac{1}{2}})$ for the
cases of $V\in L^\infty$ and $V\in C^1$. For the upper bound in
(\ref{order}), it agrees with Donnelly and Fefferman's results in
the eigenvalue case $V(x)=\lambda$.
 Recently, Kenig
\cite{K} considered a similar problem which is motivated by his work
with Bourgain in \cite{BK} on Anderson localization for the
Bernoulli model. Kenig investigated the following normalized model.
Let
\begin{equation}\lap u(x)=V(x)u(x) \quad \mbox{in} \ \mathbb B_{10}, \quad
\mbox{with} \quad \|V\|_{L^\infty}\leq M \quad \mbox{and} \quad
\|u\|_{L^\infty}\leq C_0, \label{kenig}
 \end{equation}
 where $\mathbb B_{10}$ is a ball centered at origin with radius $10$ in $\mathbb R^n$. Assume that $\sup_{|x|\leq
1}|u(x)|\geq 1$ and $M>1$. Kenig established that
\begin{equation}
\|u\|_{L^\infty (\mathbb B_r)}\geq a_1r^{a_2\beta(M)} \quad \quad
\mbox{as}\ r\to 0, \label{like}
\end{equation}
 where $a_1,
a_2$ depend only on $n$, $C_0$ and $\beta(M)$ depends on $M$. By
exploiting the Carleman estimates, Kenig proved the
$\beta(M)=M^{\frac{2}{3}}$. He also pointed out that the exponent
$\frac{2}{3}$ of $M$ is sharp for complex valued $V$ based on
Meshkov's example in \cite{M}. On the basis of Donnelly and
Fefferman's work, Kenig asked if $\beta(M)=M^{\frac{1}{2}}$ can be
achieved for real $u, V$.  Very recently, Bakri in \cite{B}
considered (\ref{schro}) in the case of $V(x)\in C^1$. He obtained
that the vanishing order of solutions in (\ref{schro}) is less than
$C(1+\sqrt{\|V\|_{C^1}})$. His proof is an extension of the Carleman
estimates in \cite{DF}. It is worthwhile to mention that the
vanishing order of solutions is closely related to the study of
eigenfunctions on manifolds. We refer to the survey \cite{Z} for
detailed account.

We are especially interested in the model (\ref{kenig}). Our first
goal in this paper is to address the above problems. Relied on a
variant of frequency function, we are able to verify that
$\beta(M)=M^{\frac{1}{2}}$ is indeed true for the case of $V(x)\in
W^{1,\infty}$ in (\ref{like}). In particular, our result also
confirms that the vanishing order of solutions in (\ref{schro}) is
less than $C(1+\|V\|_{W^{1, \infty}}^{1/2})$ if $V\in  W^{1,
\infty}$.
\begin{theorem}
Assume that $V(x)\in W^{1,\infty}(\mathbb B_{10})$. Under the
assumptions in (\ref{kenig}) with $\|V\|_{W^{1,\infty}}\leq M$, the
maximal vanishing order of $u$ in (\ref{kenig}) is less than
$C\sqrt{M}$, where $C$ depends on $n$ and $C_0$. \label{th1}
\end{theorem}

Generally speaking, the Carleman estimates and frequency function
are two principal ways to obtain quantitative uniqueness and strong
unique continuation results for solutions of partial differential
equations. Carleman estimates were introduced by Carleman, when he
studied the strong unique continuation property. Carleman estimates
are weighted integral inequalities. See e.g. \cite{H}, \cite{JK},
\cite{K}, \cite{KRS}, \cite{KT}, \cite{S}, \cite{W}, to just mention
a few. In order to obtain the quantitative uniqueness results for
solutions, one uses the Carleman estimates with a special choice of
weight functions to obtain a type of Hadamard's three-ball theorem,
then doubling estimates follow.
 The vanishing order will come from the doubling estimates. The frequency function was first observed by Almgren
 \cite{A} for
harmonic functions. The frequency function controls the local growth
rate of $u$ and is a local measure of its ``degree" as a polynomial
like function in $\mathbb B_r$. See e.g. \cite{GL}, \cite{GL1},
\cite{Lin}, \cite{HL}, \cite{Ku0}, \cite{Ku}, etc. Garofalo and Lin
in \cite{GL}, \cite{GL1} showed its powerful applications in strong
unique continuation problem. The frequency function Garofalo and Lin
investigated for equation (\ref{kenig}) is given by
\begin{equation}
N(r)=\frac{rD(r)}{H(r)} \label{fre}
\end{equation}
where $ H(r)=\int_{\partial\mathbb B_r} u^2 \,d \sigma $ and
$D(r)=\int_{\mathbb B_r}|\nabla u|^2+Vu^2 \,dx$. After one proves
certain monotonicity of $N(r)$, the doubling estimates will follow
by a standard argument. In \cite{GL}, it was shown that $e^{Cr}N(r)$
was monotone nondecreasing. However, $C$ depends on the norm of $V$.
It can not give the optimal bound for the vanishing order of
solutions. Kukavica considered almost the same frequency function in
\cite{Ku}. He was able to move the norm of $V$ away from the
exponential, but it only gave the aforementioned bound
$C(1+(\sup_{\Omega}V_+)^{\frac{1}{2}}+osc(V)^2)$ due to the
limitations of the method. Some of the limitations come from the
fact that one can not explore $H'(r)$ more because of its
integration on the boundary of balls. Instead, we consider a variant
of (\ref{fre}). See our variant of frequency function for
Schr\"{o}dinger equations in section 2 and the frequency functions
for high order elliptic equations in section 3 for the details.
First, we establish a monotonicity property of this new variant of
frequency function. Second, based on the monotonicity results, it
leads to a $L^2$-version of Hadamard's three-ball theorem, which
further implies a $L^\infty$-version of Hadamard's three-ball
theorem by elliptic estimates. At last, by a propagation of
smallness argument, we derive the vanishing order of solutions.

Higher order elliptic equations are also important models in the
study of partial differential equations. A nature question is to
study the quantitative uniqueness of higher order elliptic
equations. Our second goal is to investigate the vanishing order for
solutions of higher order elliptic equations. We consider this
normalized model:
\begin{equation}
(-\lap)^m u(x)=\overline V(x) u(x)  \ \ \mbox{in} \ \mathbb B_{10}
\quad \mbox{with} \quad \|\overline{V}\|_{L^\infty}\leq M \quad
\mbox{and} \quad \|u\|_{L^\infty}\leq C_0. \label{nor}
\end{equation}
We also assume that $\sup_{|x|\leq 1}|u(x)|\geq 1$ and $M>1$. To the
best of our knowledge, the explicit vanishing order as Theorem
\ref{th1} seems to be unknown for higher order elliptic equations.
By exploiting this variant of frequency function, we are able to
obtain the following theorem.
\begin{theorem}
Assume that $\overline V(x)\in L^\infty(\mathbb B_{10})$ and $n\geq
4m$. Under the assumption in (\ref{nor}), the maximal order of
vanishing of $u$ in (\ref{nor}) is less than $CM$, where $C$ depends
on $n$, $m$ and $C_0$. \label{th2}
\end{theorem}

Unlike the Laplacian operator in (\ref{kenig}), new difficulties
arise since some kind of ``symmetry" is lost for higher order
elliptic equations. Our idea is to break the higher order elliptic
equations into a system of semilinear equations. However, it still
does not give the most desirable result as Theorem \ref{th1}. We
also develop a $L^\infty$-version of Hadamard's three-ball theorem
by exploring $W^{2m, p}$ estimates for higher order elliptic
equations (see section 3 for the details). Compared with the
frequency function argument, it seems to more difficult to obtain
the explicit vanishing order of solutions for higher order elliptic
equations by Carleman estimates.

The quantitative uniqueness has applications in mathematical
physics. For instance, the vanishing order of solutions plays an
important role in \cite{BK} on Anderson localization for the
Bernoulli model. Suppose that $u$ is a solution of
\begin{equation}
-\lap u=Vu \quad \mbox{in} \ \ \mathbb R^n, \label{bkk}
\end{equation}
where $\|V\|_{L^\infty}\leq 1, \ \|u\|_{L^\infty}\leq C_0 \
\mbox{and} \ u(0)=1$. Let
$$ M(R)=\inf_{|x_0|=R}\sup_{\mathbb B_1(x_0)}|u(x)| $$
for $R$ large. By using the result of (\ref{like}), Bourgain and
Kenig in \cite{BK} showed that
$$ M(R)\geq C\exp(-CR^{\frac{4}{3}}\log R),$$
where $C$ depends on $n$ and $C_0$. We can consider a similar
quantitative unique continuation problem as (\ref{bkk}) for higher
order elliptic equations. Suppose that $u$ is a solution to
\begin{equation}
(-\lap)^{m} u=\overline Vu \quad \mbox{in} \ \ \mathbb R^n,
\label{bkh}
\end{equation}
where $\|\overline V\|_{L^\infty}\leq 1, \ \|u\|_{L^\infty}\leq C_0
\ \mbox{and} \ u(0)=1$. Theorem \ref{th2} implies that following
corollary for higher order elliptic equations.
\begin{corollary}
Let $u$ be a solution to (\ref{bkh}) and $n\geq 4m$. Then
$$ M(R)\geq C\exp(-CR^{2m}\log R),$$
where $C$ depends on $n$, $m$ and $C_0$. \label{cor2}
\end{corollary}

An easy consequence of Theorem \ref{th2} is a strong unique
continuation result for higher order elliptic equations when $n\geq
4m$. Due to the conclusion in Theorem \ref{th2}, the solutions will
not vanish of infinite order when $n\geq 4m$. Without using the
conclusion of Theorem \ref{th2}, we are also able to give a proof
based on this variant of frequency function.  We refer to, e.g.
\cite{CG}, \cite{LSW}, for the strong unique continuation results of
higher order elliptic equations by using Carleman estimates. Assume
that
\begin{equation}
(-\lap )^m u(x)=\overline{V}(x)u(x) \quad \quad \mbox{in} \ \Omega,
\label{fff2}
\end{equation}
where $ \overline{V}(x)\in L^{\infty}_{loc}(\Omega)$.  A function
$u\in L^2_{loc}(\Omega)$ is said to vanish of infinite order at some
point $x_0\in \Omega$ if for $R>0$ sufficiently small,
\begin{equation}
\int_{\mathbb B_{R}(x_0)}u^2 \,dx= O(R^N) \label{vani}
\end{equation}
for every positive integer $N$. We are able to establish the
following theorem.
\begin{theorem} If $u$ in (\ref{fff2}) vanishes of infinite order at some point $x_0\in \Omega$, then $u\equiv 0$ in $\Omega$.
\label{th3}
\end{theorem}

The outline of the paper is as follows. Section 2 is devoted to
obtaining the vanishing order of Schr\"{o}dinger equations. In
Section 3, the vanishing order of higher order elliptic equations is
shown. In section 4, we obtain the strong unique continuation for
higher order elliptic equations. In the whole paper, we will use
various letters, such as $C$, $D$, $E$, $K$, to denote the positive
constants which may depend $n$ and $m$, even if they are not
explicitly stated. They may also vary from line to line. Especially,
the letters do not depend on $V$ in section 2 and $\overline{V}$ in
section 3.
\section{Schr\"{o}dinger equations}

In this section, we focus on the maximal vanishing order of
solutions in (\ref{kenig}). Let $x_0\in \mathbb B_1$. We define
\begin{equation}
H_{x_0}(r)=\int_{\mathbb B_r(x_0)} u^2(r^2-|x-x_0|^2)^\alpha \,dx.
\label{nota}
\end{equation}
The value of the constant $\alpha>0$ will be determined later on. We
can assume that $\mathbb B_r(x_0)\subset \mathbb B_{10}$ by choosing
$r$ suitable small. Without loss of generality, we may assume
$x_0=0$ and denote $\mathbb B_r(0)$ as $\mathbb B_r$, that is,
$$H(r)=\int_{\mathbb
B_r} u^2(r^2-|x|^2)^\alpha \,dx. $$ The advantage of the weight
function $(r^2-|x|^2)^\alpha$ in the integration is that the
boundary term will not appear whenever we use divergence theorem.
Moreover, the value of $\alpha$ will help reduce the  order of
vanishing. The function in (\ref{nota}) appeared in \cite{Ku1} for
the study of vortex of Ginzburg-Landau equations. We will also omit
the integration on $\mathbb B_r$ when it is clear from the context.
Taking the derivative with respect to $r$ for $H(r)$, we get
\begin{eqnarray*}
 H'(r)&=&2\alpha r\int
u^2(r^2-|x|^2)^{\alpha-1}\,dx.
\end{eqnarray*}
Because of the presence of the weight function, as we mentioned
before, there are no terms involving integration on the boundary.
One of it's advantage is that it simplifies our calculations in the
following. Furthermore,
\begin{eqnarray*} H'(r)&=&\frac{2\alpha}{r}\int
u^2(r^2-|x|^2)^{\alpha}\,dx +\frac{2\alpha}{r}\int
u^2(r^2-|x|^2)^{\alpha-1}|x|^2\,dx
\\
 &=&\frac{2\alpha}{r}H(r)-\frac{1}{r}\int u^2
x\cdot\nabla(r^2-|x|^2)^{\alpha}\,dx.
\end{eqnarray*}
Applying the divergence theorem for the second term in the right
hand side of the latter equality, we get
\begin{equation}
H'(r)=\frac{2\alpha+n}{r} H(r)+\frac{1}{(\alpha+1)r}I(r),
\label{der}
\end{equation}
where
\begin{equation}
I(r)=2(\alpha+1)\int(x\cdot \nabla u) u(r^2-|x|^2)^{\alpha}\,dx.
\label{rev}
\end{equation}
Using the divergence theorem again for $I(r)$, it follows that
\begin{eqnarray}
I(r)&=&-\int u\nabla u\cdot
\nabla(r^2-|x|^2)^{\alpha+1}\,dx \nonumber \\
&=&\int |\nabla u|^2(r^2-|x|^2)^{\alpha+1}\,dx+\int V
u^2(r^2-|x|^2)^{\alpha+1}\,dx, \label{dive}
\end{eqnarray}
where we perform integration by parts and use the equation
(\ref{kenig}) in the last equality.

 We define our variant of frequency function as
\begin{equation}
N(r)=\frac{I(r)}{H(r)}.
\end{equation}
Next we are going to study the monotonicity property of this special
type of frequency function $N(r)$. We are able to obtain the
following result.

\begin{lemma}
There exists a constant $C$ depending only on $n$ such that
$$ N(r)+C \|V\|_{W^{1 ,\infty}} r^2$$
is nondecreasing function of $r\in (0, 1)$. \label{mono}
\end{lemma}
\begin{proof}
To consider the monotonicity of $N(r)$, we shall consider the
derivative of $I(r)$. By taking the derivative for $I(r)$ in
(\ref{dive}) with respect to $r$,
\begin{eqnarray*}
I'(r)=2(\alpha+1)r\int |\nabla
u|^2(r^2-|x|^2)^{\alpha}\,dx+2(\alpha+1)r\int V
u^2(r^2-|x|^2)^{\alpha}\,dx.
\end{eqnarray*}
We simply the first term in the right hand side of the latter
equality. It yields that
\begin{eqnarray*}
I'(r)&=&\frac{2(\alpha+1)}{r} \int |\nabla
u|^2(r^2-|x|^2)^{\alpha+1}\,dx-\frac{1}{r}\int x\cdot\nabla
(r^2-|x|^2)^{\alpha+1}|\nabla u|^2\,dx \\
&&+ 2(\alpha+1)r\int Vu^2(r^2-|x|^2)^\alpha\,dx.
\end{eqnarray*}
Integrating by parts for the second term in the right hand side of
the last equality gives that
\begin{eqnarray*}
 I'(r)&=&\frac{2(\alpha+1)+n}{r} \int
|\nabla u|^2(r^2-|x|^2)^{\alpha+1}\,dx+\sum_{j,l=1}^n\frac{2}{r}\int
\partial_ju
\partial_{jl}ux_l(r^2-|x|^2)^{\alpha+1} \,dx \\
&&+ 2(\alpha+1)r\int u^2V(r^2-|x|^2)^\alpha\,dx.
\end{eqnarray*}

We do further integration by parts for the second term  in the right
hand side of the last inequality with respect to $j$th derivative.
It follows that
\begin{eqnarray*}
I'(r)&=&\frac{2(\alpha+1)+n}{r} \int |\nabla
u|^2(r^2-|x|^2)^{\alpha+1}\,dx-\frac{2}{r}\int \lap u
(\nabla u\cdot x)(r^2-|x|^2)^{\alpha+1} \,dx \\
&&-\frac{2}{r}\int |\nabla
u|^2(r^2-|x|^2)^{\alpha+1}\,dx+\frac{4(\alpha+1)}{r}\int (\nabla
u\cdot x)^2(r^2-|x|^2)^\alpha\,dx\\&&+ 2(\alpha+1)r\int
Vu^2(r^2-|x|^2)^\alpha\,dx \\&=&\frac{2\alpha+n}{r} \int |\nabla
u|^2(r^2-|x|^2)^{\alpha+1}\,dx-\frac{2}{r}\int V u
(\nabla u\cdot x)(r^2-|x|^2)^{\alpha+1} \,dx \\
&&+\frac{4(\alpha+1)}{r}\int (\nabla u\cdot
x)^2(r^2-|x|^2)^\alpha\,dx+ 2(\alpha+1)r\int
Vu^2(r^2-|x|^2)^\alpha\,dx,
\end{eqnarray*}
where we have used the equation (\ref{kenig}) in the latter
equality. We want to interpret the first term in the the right hand
side of the last equality in terms of $I(r)$. In view of
(\ref{dive}), we have
\begin{eqnarray*}
I'(r)&= &\frac{2\alpha+n}{r}I(r)-\frac{2\alpha+n}{r}\int
Vu^2(r^2-|x|^2)^{\alpha+1}\,dx-\frac{2}{r}\int Vu (\nabla u\cdot
x)(r^2-|x|^2)^{\alpha+1} \,dx \\
&&+\frac{4(\alpha+1)}{r}\int (\nabla u\cdot
x)^2(r^2-|x|^2)^\alpha\,dx+2(\alpha+1)r\int
Vu^2(r^2-|x|^2)^\alpha\,dx.
\end{eqnarray*}
We breaks down the second term in the last equality as
\begin{eqnarray*}
\frac{2\alpha+n}{r}\int Vu^2(r^2-|x|^2)^{\alpha+1}\,dx
&=&(2\alpha+n)r \int Vu^2(r^2-|x|^2)^{\alpha}\,dx \\
\\&&-\frac{2\alpha+n}{r}\int Vu^2(r^2-|x|^2)^{\alpha}|x|^2\,dx.
\end{eqnarray*}
Substituting the latter equality to $I'(r)$, one obtains
\begin{eqnarray*}
I'(r)&=& \frac{2\alpha+n}{r}I(r)+(2-n)r\int
Vu^2(r^2-|x|^2)^{\alpha}\,dx+\frac{2\alpha+n}{r}\int
Vu^2(r^2-|x|^2)^{\alpha}|x|^2\,dx \\
&&+\frac{4(\alpha+1)}{r}\int (\nabla  u\cdot
x)^2(r^2-|x|^2)^\alpha\,dx-\frac{2}{r}\int Vu (\nabla u\cdot
x)(r^2-|x|^2)^{\alpha+1} \,dx.
\end{eqnarray*}
Applying the divergence theorem for the last term in the right hand
side of the latter equality and considering the fact that $V\in
W^{1, \infty}$, we arrive at
\begin{eqnarray*}
I'(r)&=& \frac{2\alpha+n}{r}I(r)+(2-n)r\int
Vu^2(r^2-|x|^2)^{\alpha}\,dx+\frac{2\alpha+n}{r}\int
Vu^2(r^2-|x|^2)^{\alpha}|x|^2\,dx \\
&&+\frac{4(\alpha+1)}{r}\int (\nabla u\cdot
x)^2(r^2-|x|^2)^\alpha\,dx+\frac{1}{r}\int (\nabla V\cdot x)
u^2(r^2-|x|^2)^{\alpha+1}\, dx \\ &&+\frac{n}{r}\int
Vu^2(r^2-|x|^2)^{\alpha+1}\,dx -\frac{2(\alpha+1)}{r}\int
Vu^2(r^2-|x|^2)^{\alpha}|x|^2\,dx.
\end{eqnarray*}
Combining the third term and seventh term in the right hand side of
the latter equality gives that
\begin{eqnarray*}
I'(r)&\geq& \frac{2\alpha+n}{r}I(r)+(2-n)r\int V
u^2(r^2-|x|^2)^{\alpha}\,dx+\frac{n-2}{r}\int
Vu^2(r^2-|x|^2)^{\alpha}|x|^2\,dx \\
&&+\frac{4(\alpha+1)}{r}\int (\nabla u\cdot
x)^2(r^2-|x|^2)^\alpha\,dx+\frac{1}{r}\int (\nabla V\cdot
x) u^2(r^2-|x|^2)^{\alpha+1}\,dx\\
&&+\frac{n}{r}\int V u^2(r^2-|x|^2)^{\alpha+1}\,dx.
\end{eqnarray*}
By the definition of $H(r)$ in (\ref{nota}) and the assumption that
$0<r<1$, we obtain
\begin{equation}
I'(r)\geq
\frac{2\alpha+n}{r}I(r)-(3n+5)r\|V\|_{W^{1,\infty}}H(r)+\frac{4(\alpha+1)}{r}\int
(\nabla u\cdot x)^2(r^2-|x|^2)^\alpha\,dx. \label{nee}
\end{equation}
In order to find the monotonicity of $N(r)$, it suffices to take the
derivative for $N(r)$ with respect to $r$. Taking $H'(r)$ in
(\ref{der}) and $I'(r)$ in (\ref{nee}) into consideration, we get
\begin{eqnarray*}
N'(r)&=&\frac{I'(r)H(r)-H'(r)I(r)}{H^2(r)} \\
&\geq&{1}/{H^2(r)}\big\{\frac{2\alpha+n}{r}I(r)H(r)-(3n+5)r\|V\|_{W^{1,\infty}}H^2(r)
\\&&+\frac{4(\alpha+1)}{r}\int
(\nabla u\cdot x)^2(r^2-|x|^2)^\alpha\,dx\int
u^2(r-|x|^2)^\alpha\,dx
-\frac{2\alpha+n}{r}I(r)H(r)\\&&-\frac{1}{r(\alpha+1)}I^2(r)\big\} \\
&\geq&
1/H^2(r)\big\{-(3n+5)r\|V\|_{W^{1,\infty}}H^2(r)-\frac{4(\alpha+1)}{r}\big(\int(x\cdot
\nabla u) u(r^2-|x|^2)^{\alpha}\,dx\big)^2\\
&&+\frac{4(\alpha+1)}{r}\int (\nabla u\cdot
x)^2(r^2-|x|^2)^\alpha\,dx\int u^2(r-|x|^2)^\alpha\,dx\big\},
\end{eqnarray*}
where we have used $I(r)$ in (\ref{rev}) in the last inequality. By
Cauchy-Schwarz inequality,  we know
$$\big(\int(x\cdot
\nabla u) u(r^2-|x|^2)^{\alpha}\,dx\big)^2\leq  \int (\nabla u\cdot
x)^2(r^2-|x|^2)^\alpha\,dx\int u^2(r-|x|^2)^\alpha\,dx.
$$
We finally arrive at
$$N'(r)\geq -(3n+5)r\|V\|_{W^{1,\infty}},$$
which implies the conclusion in the lemma.
\end{proof}

Let us compare more about our variant of frequency function and that
in \cite{GL}. Both lead to monotonicity property. Unlike the
monotonicity results in \cite{GL}, the function $V(x)$ is moved away
from the exponential in Lemma \ref{mono} . Our monotonicity result
only relies on the polynomial growth of $V(x)$. More important, the
positive position $C$ and the radius $r$ do not depend on $V$ in
Lemma \ref{mono}. The fact that $r$ is independent of $V$ is crucial
in the propagation of smallness arguments in the proof of Theorem
\ref{th1}. With the help of monotonicity of $N(r)$, we are going to
establish a $L^2$-version of Hadamard's three-ball theorem. For the
variants of Hadamard's three-ball theorem, see e.g. \cite{JL} and
\cite{Ku}. We also want to get rid of the weight function
$(r^2-|x|^2)^\alpha$ in our function $H(r)$. In this process, the
value of $\alpha$ helps reduce the coefficient in the following
three-ball theorem, which provides better vanishing order. This is
another advantage we introduce the weight function. Let
$$h(r)=\int_{\mathbb B_r(x_0)}u^2\,dx.$$ Without loss of
generality, we may assume $x_0=0$. We can easily check that
\begin{equation}
H(r)\leq r^{2\alpha}h(r) \label{kk}
\end{equation}
and
\begin{equation}
h(r)\leq \frac{H(\rho)}{(\rho^2-r^2)^\alpha} \label{kk1}
\end{equation}
for any $0<r<\rho<1$. We are able to obtain the following three-ball
theorem.
\begin{lemma}
Let $0<r_1<r_2<2r_2<r_3<1$. Then
\begin{equation}
h(r_2)\leq
\exp{(C\sqrt{M})}h^{\frac{\alpha_0}{\alpha_0+\beta_0}}(r_1)h^{\frac{\beta_0}{\alpha_0+\beta_0}}(r_3),
\label{three}
\end{equation}
where
$$ \alpha_0=\log\frac{r_3}{2r_2} $$ and $$
\beta_0=\log\frac{2r_2}{r_1}.$$ \label{niu}
\end{lemma}
\begin{proof}
From (\ref{der}), we have
\begin{equation}\frac{H'(r)}{H(r)}=\frac{2\alpha+n}{r}+\frac{1}{(\alpha+1)r}N(r).
\label{Mon}
\end{equation}
Taking integration from $2r_2$ to $r_3$ in the last identity gives
that
\begin{equation}
\log\frac{H(r_3)}{H(2r_2)}=(2\alpha+n)\log\frac{r_3}{2r_2}+\frac{1}{\alpha+1}\int^{r_3}_{2r_2}\frac{N(r)}{r}\,dx.
\label{ggg}
\end{equation}
By the monotonicity result in Lemma \ref{mono}, it follows that
$$\log\frac{H(r_3)}{H(2r_2)}\geq
(2\alpha+n)\log\frac{r_3}{2r_2}+\frac{1}{\alpha+1}(N(2r_2)+C\|V\|_{W^{1,\infty}}r_2^2)\log\frac{r_3}{2r_2}-\frac{C\|V\|_{W^{1,\infty}}}{\alpha+1}r_3^2,$$
that is,
\begin{equation}
\frac{\log\frac{H(r_3)}{H(2r_2)}+\frac{C\|V\|_{W^{1,\infty}}}{\alpha+1}r_3^2}{\log\frac{r_3}{2r_2}}\geq
(2\alpha+n)+\frac{1}{\alpha+1}(N(2r_2)+C\|V\|_{W^{1,\infty}}r_2^2).
\label{es1}
\end{equation}
If we perform similar calculations on (\ref{Mon}) by integrating
from $r_1$ to $2r_2$, we deduce that
\begin{eqnarray*}
\log\frac{H(2r_2)}{H(r_1)}&=&(2\alpha+n)\log\frac{2r_2}{r_1}+\frac{1}{\alpha+1}\int^{2r_2}_{r_1}
\frac{N(r)}{r}\,dr \\
&\leq
&(2\alpha+n)\log\frac{2r_2}{r_1}+\frac{1}{\alpha+1}\log\frac{2r_2}{r_1}(N(2r_2)+C\|V\|_{W^{1,\infty}}r_2^2).
\end{eqnarray*}
Namely,
\begin{equation}
\frac{\log\frac{H(2r_2)}{H(r_1)}}{\log\frac{2r_2}{r_1}}\leq
(2\alpha+n)+\frac{1}{\alpha+1}(N(2r_2)+C\|V\|_{W^{1,\infty}}r^2_2).
\label{es2}
\end{equation}
Combining the inequalities (\ref{es1}) and (\ref{es2}), note that
$\|V\|_{W^{1,\infty}}\leq M$, we conclude that
\begin{equation}
\frac{\log\frac{H(r_3)}{H(2r_2)}+\frac{CM}{\alpha+1}r_3^2}{\log\frac{r_3}{2r_2}}
\geq \frac{\log\frac{H(2r_2)}{H(r_1)}}{\log\frac{2r_2}{r_1}}.
\label{com}
\end{equation}
Thanks to (\ref{kk}) and (\ref{kk1}), we have
\begin{eqnarray*}
\log\frac{H(r_3)}{H(2r_2)}&\leq & \log (r_3^{2\alpha}h(r_3))-\log
(3r_2)^{\alpha}-\log h(r_2) \\
&\leq & 2\alpha \log r_3+\log h(r_3)-2\alpha \log(2r_2)-\log
h(r_2)+\alpha \log\frac{4}{3}.
\end{eqnarray*}
Therefore,
\begin{equation}
\frac{\log\frac{H(r_3)}{H(2r_2)}+\frac{CM}{\alpha+1}r_3^2}{\log\frac{r_3}{2r_2}}
\leq
2\alpha+\frac{\log\frac{h(r_3)}{h(r_2)}}{\log\frac{r_3}{2r_2}}+\frac{\alpha
\log\frac{4}{3}}{\log\frac{r_3}{2r_2}}+\frac{\frac{CM}{\alpha+1}r_3^2}{\log\frac{r_3}{2r_2}}.
\label{nna}
\end{equation}
We conduct the similar calculations as above for
$\log\frac{H(2r_2)}{H(r_1)}$ in (\ref{es2}). Using (\ref{kk}) and
(\ref{kk1}) again,
\begin{eqnarray*}
\log\frac{H(2r_2)}{H(r_1)}&\geq& \log ((3r_2^2)^{\alpha}h(r_2))-\log
r^{2\alpha}_1-\log h(r_1) \\
&\geq & 2\alpha\log(2r_2)-\alpha\log\frac{4}{3}+\log
h(r_2)-2\alpha\log r_1-\log h(r_1).
\end{eqnarray*}
So we obtain that
\begin{equation}
\frac{\log\frac{H(2r_2)}{H(r_1)}}{\log\frac{2r_2}{r_1}}\geq
2\alpha-\frac{\alpha\log\frac{4}{3}}{\log\frac{2r_2}{r_1}}+\frac{\log\frac{h(r_2)}{h(r_1)}}
{\log\frac{2r_2}{r_1}}. \label{fin}
\end{equation}
Taking (\ref{com}), (\ref{nna}) and (\ref{fin}) into account, we get
$$\frac{\log\frac{h(r_3)}{h(r_2)}}{\log\frac{r_3}{2r_2}}+\frac{\alpha
\log\frac{4}{3}}{\log\frac{r_3}{2r_2}}+\frac{\frac{CM}{\alpha+1}r_3^2}{\log\frac{r_3}{2r_2}}\geq
-\frac{\alpha\log\frac{4}{3}}{\log\frac{2r_2}{r_1}}+\frac{\log\frac{h(r_2)}{h(r_1)}}
{\log\frac{2r_2}{r_1}}.
$$
Namely,
$$(\alpha_0+\beta_0)\alpha\log\frac{4}{3}+\beta_0\big(\log\frac{h(r_3)}{h(r_2)}+\frac{CM}{\alpha+1}r_3^2\big)\geq
\alpha_0\log\frac{h(r_2)}{h(r_1)}.$$ Taking exponentials of both
sides implies that
$$h(r_2)\leq
\exp(C(\alpha+\frac{M}{\alpha+1}r_3^2))h^{\frac{{\alpha_0}}{\alpha_0+\beta_0}}(r_1)
h^{\frac{\beta_0}{\alpha_0+\beta_0}}(r_3).
$$
Note that $ 0<r_3<1$. As we know, the minimum value of the
exponential function in the last inequality will be achieved if we
take $\alpha=\sqrt{M}$. Hence
$$h(r_2)\leq
\exp{(C\sqrt{M})}h^{\frac{\alpha_0}{\alpha_0+\beta_0}}(r_1)h^{\frac{\beta_0}{\alpha_0+\beta_0}}(r_3),$$
where $C$ is a constant depending only on $n$. We are done with the
$L^2$-version of three-ball theorem.
\end{proof}

From the above lemma, one can see that the appearance of $\alpha$
reduces the exponent of exponential in the $L^2$-version of
three-ball theorem. Thanks to Lemma \ref{niu}, we are able to
establish a $L^\infty$-version of three-ball theorem, which will be
used in the propagation of smallness argument.
\begin{lemma}
Let $0<r_1<r_2<2r_2<r_3<1$. Then
\begin{equation} \|u\|_{L^\infty(\mathbb B_{r_2})}\leq
C\exp{(C\sqrt{M})}(\frac{r_3^2}{r_3-2r_2})^{\frac{n}{2}}
\|u\|_{L^\infty(\mathbb
B_{r_1})}^{\frac{\alpha_1}{\alpha_1+\beta_1}}
\|u\|_{L^\infty(\mathbb
B_{r_3})}^{\frac{\beta_1}{\alpha_1+\beta_1}}, \label{infy}
\end{equation}
where
$$\alpha_1=\log\frac{r_3}{\frac{2}{3}(r_2+r_3)}$$ and
$$\beta_1=\log\frac{\frac{2}{3}(r_2+r_3)}{r_1}.$$
\label{kao}
\end{lemma}
\begin{proof}
Using the standard elliptic theory for the solution in
$(\ref{kenig})$, we have
\begin{equation}
\|u\|_{L^\infty(\mathbb B_{\delta})}\leq
C(\|V\|_{L^\infty}+1)^{\frac{n}{2}}\delta^{-\frac{n}{2}}\|u\|_{L^2(\mathbb
B_{2\delta})}, \label{ttt}
\end{equation}
here $C$ does not depends on $\delta$. By some rescaling argument,
$$\|u\|_{L^\infty(\mathbb B_{r})}\leq
C(\|V\|_{L^\infty}+1)^{\frac{n}{2}}(\rho-r)^{-\frac{n}{2}}\|u\|_{L^2(\mathbb
B_{\rho})}$$ for $0<r<\rho<1$. Then
$$\|u\|_{L^\infty(\mathbb B_{r_2})}\leq
C(\|V\|_{L^\infty}+1)^{\frac{n}{2}}(r_3-2r_2)^{-\frac{n}{2}}\|u\|_{L^2(\mathbb
B_{\frac{r_2+r_3}{3}})}.$$ Taking advantage of Lemma \ref{niu}, we
deduce that
$$\|u\|_{L^\infty(\mathbb B_{r_2})}\leq C
(\|V\|_{L^\infty}+1)^{\frac{n}{2}}(r_3-2r_2)^{-\frac{n}{2}}r_3^n\exp{(C\sqrt{M})}
\|u\|_{L^\infty(\mathbb B(r_1))}^{\frac{\alpha_1}{\alpha_1+\beta_1}}
\|u\|_{L^\infty(\mathbb
B(r_3))}^{\frac{\beta_1}{\alpha_1+\beta_1}}.$$ Recall that
$\|V\|_{W^{1, \infty}}\leq M$. Thus, we arrive at the conclusion.
\end{proof}

Now we are ready to prove Theorem \ref{th1}. We apply the idea of
propagation of smallness  which is based on overlapping of
three-ball argument. Similar arguments have been employed in
\cite{DF}.
\begin{proof} [Proof of Theorem  \ref{th1}]
We choose a small $r$ such that $$\sup_{\mathbb
B_{\frac{r}{2}(0)}}|u|=\epsilon.$$ Obviously, $\epsilon>0$. Since
$\sup_{|x|\leq 1}|u(x)|\geq 1$, there exists some $\bar x\in \mathbb
B_1$ such that $u(\bar x)=\sup_{|x|\leq 1}|u(x)|\geq 1$. We select a
sequence of balls with radius $r$ centered at $x_0=0, \ x_1, \cdots,
x_d$ so that $x_{i+1}\in \mathbb B_{\frac{r}{2}}(x_i)$ and $\bar
x\in \mathbb B_{r}(x_d)$, where $d$ depends on the radius $r$ which
we will fix later on.  Employing Lemma \ref{kao} with
$r_1=\frac{r}{2}$, $r_2=r$, and $r_3=3r$ and the boundedness
assumption of $u$, we get
$$ \|u\|_{L^\infty{(\mathbb B_r)}}\leq C_1\epsilon^{\theta}\exp{(C\sqrt{M})} $$
where $1<\theta=\frac{\log\frac{9}{8}}{\log 6}<1$ and $C_1$ depends
on the $L^\infty$-norm of $u$.

Iterating the above argument with Lemma \ref{kao} for balls centered
at $x_i$ and using the fact that $\|u\|_{L^\infty(\mathbb
B_\frac{r}{2}(x_{i+1}))}\leq \|u\|_{L^\infty(\mathbb B_r(x_{i}))}$,
we have
$$\|u\|_{L^\infty(\mathbb B_{r}(x_{i}))}\leq C_i
\epsilon^{D_i}\exp{(E_i\sqrt{M})}$$ for $i=0, 1, \cdots, d$, where
$C_i$ is a constant depending on $d$ and $L^\infty$-norm of $u$, and
$D_i$, $E_i$ are constants depending on $d$. By the fact that $
u(\bar x)\geq 1$ and $\bar x \in \mathbb B_{r}(x_d)$, we obtain
$$\sup_{\mathbb B_{\frac{r}{2}(0)}}|u|=\epsilon\geq
K_1\exp{(-K_2{\sqrt{M}})},$$ where $K_1$ is a constant  depending on
$d$ and $L^\infty$-norm of $u$, and $K_2$ is a constant depending on
$d$.

Applying the $L^\infty$ type of three-ball lemma again centered at
origin again with $r_2=\frac{r}{2}>r_1$ and $r_3=3r$, where $r_1$ is
sufficiently small, we have
$$ K_1\exp{(-K_2{\sqrt{M}})}\leq C\exp{(C\sqrt{M})}
\|u\|_{L^\infty(\mathbb
B_{r_1})}^{\frac{\alpha_1}{\alpha_1+\beta_1}}
C_0^{\frac{\beta_1}{\alpha_1+\beta_1}}.
$$
Recall that $C_0$ is the $L^\infty$ norm of $u$ in $ \mathbb
B_{10}$. Then
\begin{equation}
K_3^{1+q}\exp{(-(1+q)K_4{\sqrt{M}})}\leq \|u\|_{L^\infty(\mathbb
B_{r_1})}, \label{fina2}
\end{equation}
 where $K_3$ depends on $d$ and $L^\infty$ norm of $u$, $K_4$ depends on
 $d$, and
$$ q=\frac{\beta_1}{\alpha_1}=\frac{\log\frac{7}{3}r- \log r_1}{\log\frac{9}{7}}= -K_5+\log\frac{9}{7}\log\frac{1}{r_1}$$
with constant $K_5>0$ depending on $r$. Now we can fix the small
$r$. For instance, let $r=\frac{1}{100}$. Thus, the number $d$ is
also determined. The inequality (\ref{fina2}) implies that
$$\|u\|_{L^\infty(\mathbb B_{r_1})}\geq K_6r_1^{K_7\sqrt{M}}, $$
where the constants $K_6, K_7$ depend on the dimension $n$ and
$C_0$. Therefore, Theorem 1 is completed.
\end{proof}

\section{Higher order elliptic equations}
In this section, we consider the vanishing order of solutions for
the higher order elliptic equations. As far as we know, the explicit
vanishing order seems to be unknown in the literature. Due to the
complexity of its structure, we decompose the model in (\ref{nor})
into a system of $m$ semilinear equations, that is,
\begin{equation}
\left \{\begin{array}{lll} -\lap u_1&=&u_2, \\
-\lap u_i&=&u_{i+1}, \quad i=2,\cdots, m-1,\\
-\lap u_m&=& \overline Vu_1.
\end{array}
\right. \label{syst}
\end{equation}
Note that $u_1=u$. Inspired by our frequency function in section 2,
it is nature to consider the following function for the system of
semilinear equations in (\ref{syst}). Let \begin{equation}
H_{x_0}(r)=\sum^m_{i=1}\int_{\mathbb
B_r(x_0)}u_i^2(r^2-|x-x_0|^2)^{\alpha}\,dx.
\label{hhl}\end{equation} As before, we may assume $x_0=0$ and omit
the integration on $\mathbb B_r$ if it is clear from the context.
Namely,
$$H(r)=\sum^m_{i=1}\int u_i^2(r^2-|x|^2)^{\alpha} \,dx. $$ The value of the constant $\alpha>0$ will
be determined later on. If one takes derivative for $H(r)$ with
respect to $r$, following the similar calculations in section 2, one
has
\begin{eqnarray}
H'(r)&=&2\alpha r\sum^m_{i=1}\int u_i^2(r^2-|x|^2)^{\alpha-1}\,dx
\nonumber\\
&=&\frac{2\alpha}{r}\sum^m_{i=1}\int
u_i^2(r^2-|x|^2)^\alpha\,dx+\frac{2\alpha}{r}\sum^m_{i=1}\int
u_i^2(r^2-|x|^2)^{\alpha-1}|x|^2\,dx \nonumber \\
&=&\frac{2\alpha}{r}H(r)-\frac{1}{r}\sum^m_{i=1}\int u_i^2
x\cdot\nabla(r^2-|x|^2)^{\alpha}\,dx.
 \label{rege}
\end{eqnarray}
Performing the divergence theorem for the second term in the right
hand side of the last equality, we obtain that
\begin{equation}
H'(r)=\frac{2\alpha+n}{r}H(r)+\frac{1}{(\alpha+1)r}I(r),
\label{come}
\end{equation}
where
\begin{equation}I(r)=2(\alpha+1)\sum^m_{i=1}\int (x\cdot\nabla
u_i)u_i(r^2-|x|^2)^{\alpha} \,dx. \label{inb}\end{equation} Applying
the divergence theorem on $I(r)$, we have
\begin{eqnarray}
I(r)&=&-\sum^m_{i=1}\int \nabla u_i
\cdot\nabla(r^2-|x|^2)^{\alpha+1} u_i\,dx \nonumber \\
&=&\sum^m_{i=1}\int |\nabla u_i|^2
(r^2-|x|^2)^{\alpha+1}\,dx+\sum^m_{i=1}\int \lap u_i
u_i(r^2-|x|^2)^{\alpha+1} \,dx.\label{lasth}
\end{eqnarray}
Considering the systems of equations (\ref{syst}), it follows that
\begin{eqnarray}
I(r)&=&\sum^m_{i=1}\int |\nabla
u_i|^2(r^2-|x|^2)^{\alpha+1}\,dx-\sum^{m-1}_{i=1}\int
u_{i+1}u_i(r^2-|x|^2)^{\alpha+1}\,dx \label{sui} \\&&-\int
\overline{V}u_mu_1(r^2-|x|^2)^{\alpha+1}\,dx. \nonumber
\end{eqnarray}

For the higher order elliptic equations, we define our variant of
frequency function as
\begin{equation}
N(r)=\frac{I(r)}{H(r)}.
\end{equation}
Since we are dealing with more complex structure, more careful
calculations are devoted. Different from the semilinear equation
case,  higher regularity, i.e. $\overline{V}(x)\in W^{1,\, \infty}$
seems not be helpful. We consider the case that $\overline{V}(x)\in
L^\infty$. We are able to obtain the following the monotonicity
property for the frequency function $N(r)$.

\begin{lemma}
There exists a constant $C$ depending only on $n, m$ such that
$$\exp{(Cr)}\big(N(r)+\alpha(\|\overline{V}\|_{L^\infty}+1)+(\|\overline{V}\|_{L^\infty}+1)^2\big) $$
is nondecreasing function of $r\in (0, 1)$. \label{mono2}
\end{lemma}
\begin{proof}
To obtain the monotonicity result, we shall consider the derivative
of $I(r)$. Now differentiating $I(r)$ in (\ref{lasth}) with respect
to $r$,
\begin{eqnarray*}
I'(r)&=&\frac{2(\alpha+1)}{r}\sum^{m}_{i=1}\int |\nabla
u_i|^2(r^2-|x|^2)^{\alpha+1}\,dx-\frac{1}{r}\sum^{m}_{i=1}\int
|\nabla u_i|^2 \nabla(r^2-|x|^2)^{\alpha+1}\cdot x\,dx \\
&&+2(\alpha+1)r\sum^m_{i=1}\int \lap u_i u_i(r^2-|x|^2)^{\alpha}
\,dx.
\end{eqnarray*}
Integrating by parts for the second term in the right hand side of
the latter equality,
\begin{eqnarray*}
I'(r)&=&\frac{2(\alpha+1)+n}{r}\sum^{m}_{i=1}\int |\nabla
u_i|^2(r^2-|x|^2)^{\alpha+1}\,dx \\&&
+\frac{2}{r}\sum^{m}_{i=1}\sum^{n}_{l=1}\int
\partial_{jl}u_i \partial_j u_i x_l(r^2-|x|^2)^{\alpha+1}\,dx \\
&&+2(\alpha+1)r\sum^m_{i=1}\int \lap u_i u_i(r^2-|x|^2)^{\alpha}
\,dx.
\end{eqnarray*}
If one performs the divergence theorem with respect to $j$th
derivative on the second term in the right hand side of the last
inequality, one has
\begin{eqnarray*}
I'(r)&=&\frac{2(\alpha+1)+n}{r}\sum^{m}_{i=1}\int |\nabla
u_i|^2(r^2-|x|^2)^{\alpha+1}\,dx \\ &&-\frac{2}{r}\sum^{m}_{i=1}\int
(\nabla u_i\cdot x) \lap u_i (r^2-|x|^2)^{\alpha+1}\,dx
-\frac{2}{r}\sum^{m}_{i=1}\int |\nabla
u_i|^2(r^2-|x|^2)^{\alpha+1}\,dx
\\ &&+
\frac{4(\alpha+1)}{r}\sum^{m}_{i=1}\int (\nabla u_i\cdot
x)^2(r^2-|x|^2)^\alpha \,dx+2(\alpha+1)r\sum^m_{i=1}\int \lap u_i
u_i(r^2-|x|^2)^{\alpha} \,dx.
\end{eqnarray*}
Using the equivalent system of equations in (\ref{syst}), it follows
that
\begin{eqnarray*}
I'(r)&=&\frac{2\alpha+n}{r}\sum^{m}_{i=1}\int |\nabla
u_i|^2(r^2-|x|^2)^{\alpha+1}\,dx+\frac{2}{r}\sum^{m-1}_{i=1}\int
(\nabla u_i\cdot x) u_{i+1} (r^2-|x|^2)^{\alpha+1}\,dx
\\&&+\frac{2}{r}\int (\nabla u_m \cdot x)\overline{V}u_1
(r^2-|x|^2)^{\alpha+1}\,dx+\frac{4(\alpha+1)}{r}\sum^{m}_{i=1}\int
(\nabla u_i \cdot x)^2 (r^2-|x|^2)^\alpha\,dx\\
&&-2(\alpha+1)r\sum^{m-1}_{i=1}\int  u_{i+1} u_i
(r^2-|x|^2)^\alpha\,dx-2(\alpha+1)r\int
\overline{V}u_1u_m(r^2-|x|^2)^\alpha\,dx.
\end{eqnarray*}
We want to transform the first term in the right hand side of the
latter inequality in term of $I(r)$. Taking (\ref{sui}) into
consideration and performing some calculations, we have
\begin{eqnarray*}
I'(r)&=&\frac{2\alpha+n}{r}I(r)+(n-2)r \sum^{m-1}_{i=1}\int
u_{i+1}u_i(r^2-|x|^2)^\alpha \,dx \\ &&+(n-2)r\int \overline{V}u_m
u_1(r^2-|x|^2)^\alpha\,dx -\frac{2\alpha+n}{r}\sum^{m-1}_{i=1}\int
u_{i+1}u_i(r^2-|x|^2)^\alpha|x|^2\,dx \\ &&-\frac{2\alpha+n}{r}\int
\overline{V} u_m u_1(r^2-|x|^2)^\alpha|x|^2\,dx
+\frac{4(\alpha+1)}{r}\sum^{m}_{i=1}\int (\nabla u_i\cdot x)^2
(r^2-|x|^2)^\alpha\,dx \\&&+\frac{2}{r}\sum^{m-1}_{i=1}\int (\nabla
u_i\cdot  x) u_{i+1} (r^2-|x|^2)^{\alpha+1}\,dx +\frac{2}{r}\int
(\nabla u_m\cdot x)\overline{V}u_1 (r^2-|x|^2)^{\alpha+1}\,dx.
\end{eqnarray*}
Now we estimate each term in the right hand side of the last
equality. Using H\"{o}lder's inequality and the definition of $H(r)$
in (\ref{hhl}), we obtain
\begin{eqnarray}
(n-2)r \sum^{m-1}_{i=1}\int u_{i+1}u_i
(r^2-|x|^2)^\alpha\,dx&+&(n-2)r\int \overline{V}u_m
u_1(r^2-|x|^2)^\alpha\,dx \nonumber \\ &\geq&- C
r(\|\overline{V}\|_{L^\infty}+1) H(r) \label{equ1}
\end{eqnarray}
and
\begin{eqnarray}
-\frac{2\alpha+n}{r}\sum^{m-1}_{i=1}\int
u_{i+1}u_i(r^2-|x|^2)^{\alpha}|x|^2\,dx &-&\frac{2\alpha+n}{r}\int
\overline{V} u_m u_1(r^2-|x|^2)^\alpha|x|^2\,dx \nonumber\\ &&\geq-
(2\alpha+n) r(\|\overline{V}\|_{L^\infty}+1) H(r) \label{equ2}.
\end{eqnarray}
Similarly, by H\"{o}lder's inequality,
\begin{eqnarray}
\frac{2}{r}\sum^{m-1}_{i=1}\int (\nabla u_i \cdot x) u_{i+1}
(r^2-|x|^2)^{\alpha+1}\,dx&\geq& -2\sum^{m-1}_{i=1} \int |\nabla
u_i|^2(r^2-|x|^2)^{\alpha+1}\,dx \nonumber \\&&
-2\sum^{m-1}_{i=1}\int u^2_{i+1}(r^2-|x|^2)^{\alpha+1}\,dx
\label{fun1}
\end{eqnarray}and
\begin{eqnarray}
\frac{2}{r}\int (\nabla u_m \cdot x)\overline{V} u_1
(r^2-|x|^2)^{\alpha+1}\,dx &\geq &-2\int |\nabla
u_m|^2(r^2-|x|^2)^{\alpha+1}\,dx \nonumber
\\&&-\|\overline{V}\|_{L^\infty}^2\int
u_1^2(r^2-|x|^2)^{\alpha+1}\,dx. \label{fun2}
\end{eqnarray}
For the ease of the notation, let
$$v=\|\overline{V}\|_{L^\infty}+1.$$
Combining the inequalities (\ref{fun1}) and (\ref{fun2}) and taking
(\ref{sui}) into account, we get
\begin{eqnarray}
\frac{2}{r}\int \big(\sum^{m-1}_{i=1}(\nabla u_i\cdot  x) u_{i+1}&+&
  (x\cdot\nabla u_m) \overline{V}u_1\big)(r^2-|x|^2)^{\alpha+1}\,dx \nonumber
  \\ &\geq &  -C\sum^{m}_{i=1}
\int |\nabla u_i|^2(r^2-|x|^2)^{\alpha+1}\,dx \nonumber \\&&- Cv^2
\sum^{m}_{i=1}\int |u_i|^2(r^2-|x|^2)^{\alpha+1}\,dx
 \nonumber\\
&\geq &-CI(r)-Cv^2rH(r) \nonumber \\&& -C\int
(\sum^{m-1}_{i=1}u_{i+1}u_i
+\overline{V}u_mu_1)(r^2-|x|^2)^{\alpha+1}\,dx
 \nonumber \\ &\geq &
-CI(r)-Cv^2rH(r).
 \label{shi}
\end{eqnarray}
Therefore, together with (\ref{equ1}), (\ref{equ2}), and
(\ref{shi}),
\begin{eqnarray*}
I'(r)&\geq& \frac{2\alpha+n}{r}I(r)-CI(r)-C(\alpha v+v^2)H(r)
\nonumber \\&&+\frac{4(\alpha+1)}{r}\sum^{m}_{i=1}\int (\nabla
u_i\cdot x)^2(r^2-|x|^2)^{\alpha} \,dx,
\end{eqnarray*}
where $C$ depends only on $n$ and $m$. In order to get monotonicity
of the frequency function, we differentiate $N(r)$. Recall $H(r)$ in
(\ref{hhl}) and $H'(r)$ in (\ref{come}).
\begin{eqnarray*}
N'(r)&=&\frac{I'(r)H(r)-H'(r)I(r)}{H^2(r)} \\
&\geq&
\frac{1}{H^2(r)}\big\{\frac{2\alpha+n}{r}I(r)H(r)-CI(r)H(r)-C(\alpha
v+v^2)H^2(r)
\\&&+\frac{4(\alpha+1)}{r}\sum^{m}_{i=1}\int
(\nabla u_i\cdot x)^2(r^2-|x|^2)^{\alpha}\,dx\sum^{m}_{i=1}\int
|u_i|^2(r^2-|x|^2)^\alpha\,dx \\&&-\frac{2\alpha+n}{r}I(r)H(r)
-\frac{1}{(\alpha+1)r}I^2(r)\big\} \\
&\geq&\frac{1}{H^2(r)}\big\{ \frac{4(\alpha+1)}{r}\sum^{m}_{i=1}\int
(\nabla u_i\cdot x)^2(r^2-|x|^2)^\alpha\,dx\sum^{m}_{i=1}\int
|u_i|^2(r^2-|x|^2)^{\alpha}\,dx
\\&&-\frac{4(\alpha+1)}{r}\sum^{m}_{i=1}\int (\nabla u_i\cdot x)
u_i(r^2-|x|^2)^\alpha\,dx -CI(r)H(r)-C(\alpha v+v^2)H^2(r)\big\},
\end{eqnarray*}
where we have used $I(r)$ in (\ref{inb}) in the last inequality. By
Cauchy-Schwartz inequality,
$$
N'(r)+CN(r)+C(\alpha v+v^2)\geq 0.$$ Consequently,
$$ \exp{(Cr)}(N(r)+\alpha v+v^2) \ \ \mbox{is
nondecreasing.}$$ We complete the proof of the lemma.
\end{proof}
As the conclusion in Lemma \ref{mono2} indicates, the monotonicity
property only relies on the polynomial growth of $\overline{V}$ and
$r$ does not depend on $\overline V$. We are going to establish a
$L^2$-version of three-ball theorem. For convenience, let
$$
\overline{N}(r)=\exp{(Cr)}(N(r)+\alpha v+v^2).$$ We also need to
remove the weight function $(r^2-|x|^2)^{\alpha}$ in $H(r)$. As in
the section 2, let
$$h(r)=\sum^m_{i=1}\int_{B_r(x_0)}u_i^2\,dx.$$
As usual, we will omit the dependent of the center of $x_0$ for the
ball. It is easy to check that
\begin{equation}
H(r)\leq r^{2\alpha}h(r) \label{tian}
\end{equation}
and
\begin{equation}
h(r)\leq \frac{H(\rho)}{(\rho^2-r^2)^\alpha} \label{tian1}
\end{equation}
for any $0<r<\rho<1$.

Based on the monotonicity of $N(r)$ in the last lemma, we are able
to establish the following $L^2$-type of three-ball theorem.
\begin{lemma}
Let $0<r_1<r_2<2r_2<r_3<1$. Then
$$ h(r_2)\leq
(\frac{r_3}{2r_2})^{CM}\exp{(CM)}h^{\frac{\alpha_2}{\alpha_2+\beta_2}}(r_1)h^{\frac{\beta_2}{\alpha_2+\beta_2}}(r_3)
$$
where
$$\alpha_2=\log\frac{r_3}{2r_2}$$
and
$$\beta_2=C\log\frac{2r_2}{r_1},$$
where $C$ depends only on $n$ and $m$. \label{hhh}
\end{lemma}

\begin{proof}
From (\ref{come}), we deduce that
\begin{equation}\frac{H'(r)}{H(r)}=\frac{2\alpha+n}{r}+\frac{1}{(\alpha+1)r}N(r).
\label{long}
\end{equation}

 On one hand, integrating from
$r_1$ to $2r_2$ on the equality (\ref{long}) gives that
\begin{eqnarray*}
\log \frac{H(2r_2)}{H(r_1)} &=&
(2\alpha+n)\log\frac{2r_2}{r_1}+\frac{1}{\alpha+1}\int^{2r_2}_{r_1}\frac{N(r)}{r}\,dr \\
&\leq&(2\alpha+n)\log\frac{2r_2}{r_1}+\frac{\overline{N}(2r_2)}{\alpha+1}\log
\frac{2r_2}{r_1}-\frac{\alpha v+v^2}{\alpha+1}\log\frac{2r_2}{r_1},
\end{eqnarray*}
where we have used Lemma \ref{mono2} in the last inequality. Namely,
\begin{equation}
\frac{\log{H(2r_2)}/{H(r_1)}}{\log{2r_2}/{r_1}}-(2\alpha
+n)+\frac{\alpha v+v^2}{\alpha+1}\leq
\frac{\overline{N}(2r_2)}{\alpha+1}. \label{new}
\end{equation}
On the other hand, integrating from $2r_2$ to $r_3$ on the equality
(\ref{long}) implies that
\begin{eqnarray*}
\log \frac{H(r_3)}{H(2r_2)} &=&
(2\alpha+n)\log\frac{r_3}{2r_2}+\frac{1}{\alpha+1}\int^{r_3}_{2r_2}\frac{N(r)}{r}\,dr \\
&\geq&
(2\alpha+n)\log\frac{r_3}{2r_2}+\frac{1}{\alpha+1}\exp{(-C)}\overline{N}(2r_2)\log
\frac{r_3}{2r_2}-\frac{\alpha v+v^2}{\alpha+1}\log\frac{r_3}{2r_2},
\end{eqnarray*}
that is,
\begin{equation}
\frac{\log{H(r_3)}/{H(2r_2)}}{\log{r_3}/{2r_2}}-(2\alpha+n)+\frac{\alpha
v+v^2}{\alpha+1}\geq \frac{1}{\alpha+1}\exp{(-C)}\overline{N}(2r_2).
\label{new1}
\end{equation}
Taking (\ref{new}) and (\ref{new1}) into considerations, we get
\begin{equation} \frac{\log{H(2r_2)}/{H(r_1)}}{\log{2r_2}/{r_1}}\leq
C\frac{\log{H(r_3)}/{H(2r_2)}}{\log{r_3}/{2r_2}}+C\frac{\alpha
v+v^2}{\alpha+1}. \label{eric}
\end{equation}
Thanks to (\ref{tian}) and (\ref{tian1}),
\begin{eqnarray*}
\log\frac{H(r_3)}{H(2r_2)}&\leq & \log (r_3^{2\alpha}h(r_3))-\log
(3r_2)^{\alpha}-\log h(r_2) \\
&\leq & 2\alpha \log r_3+\log h(r_3)-2\alpha \log(2r_2)-\log
h(r_2)+\alpha \log\frac{4}{3}.
\end{eqnarray*}
Therefore,
\begin{equation}
\frac{\log\frac{H(r_3)}{H(2r_2)}}{\log\frac{r_3}{2r_2}}+\frac{\alpha
v+v^2}{\alpha+1} \leq
\alpha+\frac{\log\frac{h(r_3)}{h(r_2)}}{\log\frac{r_3}{2r_2}}+\frac{\alpha
\log\frac{4}{3}}{\log\frac{r_3}{2r_2}}+\frac{\alpha
v+v^2}{\alpha+1}. \label{ren1}
\end{equation}
We do the similar calculations for $\log\frac{H(2r_2)}{H(r_1)}$.
Using (\ref{tian}) and (\ref{tian1}) again,
\begin{eqnarray*}
\log\frac{H(2r_2)}{H(r_1)}&\geq& \log ((3r_2^2)^{\alpha}h(r_2))-\log
r^{2\alpha}_1-\log h(r_1) \\
&\geq & 2\alpha\log(2r_2)-\alpha\log\frac{4}{3}+\log
h(r_2)-2\alpha\log r_1-\log h(r_1).
\end{eqnarray*}
Thus,
\begin{equation}
\frac{\log\frac{H(2r_2)}{H(r_1)}}{\log\frac{2r_2}{r_1}}\geq
2\alpha-\frac{\alpha\log\frac{4}{3}}{\log\frac{2r_2}{r_1}}+\frac{\log\frac{h(r_2)}{h(r_1)}}
{\log\frac{2r_2}{r_1}}. \label{ren}
\end{equation}
Taking (\ref{eric}), (\ref{ren1}) and (\ref{ren}) into account, we
have
$$C\frac{\log\frac{h(r_3)}{h(r_2)}}{\log\frac{r_3}{2r_2}}+C\frac{\alpha
\log\frac{4}{3}}{\log\frac{r_3}{2r_2}}+C(\alpha+\frac{\alpha
v+v^2}{\alpha+1})\geq
-\frac{\alpha\log\frac{4}{3}}{\log\frac{2r_2}{r_1}}+\frac{\log\frac{h(r_2)}{h(r_1)}}
{\log\frac{2r_2}{r_1}}.
$$
Namely,
$$(\alpha_2+\beta_2)\alpha\log\frac{4}{3}+\beta_2\log\frac{h(r_3)}{h(r_2)}+
\alpha_2\beta_2(\alpha+\frac{\alpha v+v^2}{\alpha+1})\geq
\alpha_2\log\frac{h(r_2)}{h(r_1)}.$$ Taking exponentials of both
sides and performing some simplications, we obtain
$$h(r_2)\leq
\exp\big(\frac{\alpha_2\beta_2}{\alpha_2+\beta_2}(\alpha+\frac{\alpha
v+v^2}{\alpha+1})\big)\exp{(\alpha)}h^{\frac{{\alpha_2}}{\alpha_2+\beta_2}}(r_1)
h^{\frac{\beta_2}{\alpha_2+\beta_2}}(r_3).
$$
 Since
$$\frac{\alpha_2\beta_2}{\alpha_2+\beta_2}\leq
\log\frac{r_3}{2r_2},$$ we have
$$h(r_2)\leq
(\frac{r_3}{2r_2})^{(\alpha+\frac{\alpha
v+v^2}{\alpha+1})}\exp{(\alpha)}h^{\frac{{\alpha_2}}{\alpha_2+\beta_2}}(r_1)
h^{\frac{\beta_2}{\alpha_2+\beta_2}}(r_3).
$$
As we know, the minimum value of the function $\alpha+\frac{\alpha
v+v^2}{\alpha+1}$ is achieved in the case of $\alpha=v$. Recall that
$v=\|\overline{V}\|_{L^\infty}+1\leq 2M$. Therefore, the lemma is
completed.

\end{proof}

Again we need to establish a $L^\infty$-version of  three-ball
theorem. However, the classical elliptic estimates as (\ref{ttt})
does not seem to be known for higher order elliptic equations in the
literature. We will deduce a similar estimate by Sobolev inequality
and a $W^{2m,p}$ type estimate. We first present a $W^{2m,p}$ type
estimates for higher order elliptic equations (see e.g. \cite{LWZ}).
Let $u$ satisfy the following equation
\begin{equation}
(-\lap)^m u=g(x)  \quad \quad \mbox{in} \ \mathbb B \label{tts}.
\end{equation}
Then we have
\begin{lemma}
 Let
$1<p<\infty$. Suppose $u\in W^{2m, p} $ satisfies (\ref{tts}). Then
there exits a constant $C>0$ depending only on $ n, m $ such that
for any $\sigma\in (0, 1)$,
\begin{equation}
\|u\|_{W^{2m,p}( \mathbb B_{\sigma })}\leq
\frac{C(n,m)}{(1-\sigma)^{2m}}\big(\|g\|_{L^p(\mathbb
B)}+\|u\|_{L^p(\mathbb B)}\big). \label{apr}
\end{equation}
\label{afa}
\end{lemma}

Upon a rescaling argument, we have \begin{equation}
 \|u\|_{W^{2m,p}( \mathbb
B_{\sigma R })}\leq
\frac{C(n,m)}{(1-\sigma)^{2m}R^{2m}}\big(R^{2m}\|g\|_{L^p(\mathbb
B_R)}+\|u\|_{L^p(\mathbb B_R)}\big) \label{rrr}
\end{equation}
for $0<R<1$.

Applying Lemma \ref{afa}, we are able to establish  the
$L^\infty$-version of  three-ball theorem for the solutions in
(\ref{nor}).
\begin{lemma}
Let $0<r_1<r_2<4r_2<r_3<1$ and $n\geq 4m$. Then
\begin{equation} \|u\|_{L^\infty(\mathbb B_{r_2})}\leq C
\exp{(CM)}(r_3-4r_2)^{-\frac{n}{2}}(\frac{3r_3}{2(2r_2+r_3)})^{CM}
\|u\|_{L^\infty(\mathbb
B_{r_1})}^{\frac{\alpha_3}{\alpha_3+\beta_3}}
\|u\|_{L^\infty(\mathbb B_{r_3})}^{\frac{\beta_3}{\alpha_3+\beta_3}}
\label{not}
\end{equation}
where
$$\alpha_3=\log\frac{3r_3}{2(2r_2+r_3)}$$ and
$$\beta_3=C\log\frac{2r_2+r_3}{3r_1}.$$
\label{tir2}
\end{lemma}

\begin{proof}
Thanks to Lemma \ref{afa} in the case of $p=2$, we can estimate the
solution in (\ref{nor}) by the following
$$\|u\|_{W^{2m,2}( \mathbb B_{\sigma
})}\leq
\frac{C}{(1-\sigma)^{2m}}(\|\overline{V}\|_{L^\infty}+1)\|u\|_{L^2(\mathbb
B)}.$$ By Sobolev imbedding inequality, if $n>4m$,
$$ \|u\|_{L^{\frac{2n}{n-4m}}( \mathbb B_{\sigma })} \leq
\frac{C(\sigma)}{(1-\sigma)^{2m}}(\|\overline{V}\|_{L^\infty}+1)\|u\|_{L^2(\mathbb
B)},$$ where $C(\sigma)$ depends on $\sigma$, $n$ and $m$. Applying
Lemma \ref{afa} again with $p=\frac{2n}{n-4m}$ and the latter
inequality, note that $p=\frac{2n}{n-4m}>2$,
\begin{eqnarray*}
\|u\|_{W^{2m,\frac{2n}{n-4m}}( \mathbb B_{\sigma^2 })}&\leq&
\frac{C(\sigma)}{(1-\sigma)^{2m}}(\|\overline{V}\|_{L^\infty}+1)\|u\|_{L^\frac{2n}{n-4m}(\mathbb B_\sigma)}\\
&\leq &
\frac{C(\sigma)}{(1-\sigma)^{4m}}(\|\overline{V}\|_{L^\infty}+1)^2\|u\|_{L^2(\mathbb
B)}.
\end{eqnarray*}
As we know $$\|u\|_{W^{2m, q}( \mathbb B)}\geq
C\|u\|_{L^\infty(\mathbb B)}, \quad \mbox{if}  \quad
q>\frac{n}{2m}.$$ Employing the above bootstrap argument finite
times, e.g. $k$ times, which depends only on $n$ and $m$ and using
the above Sobolev imbedding inequality, we get
$$\|u\|_{L^\infty( \mathbb B_{\sigma^{k}})}\leq
\frac{C(\sigma)}{(1-\sigma)^{2^km}}(\|\overline{V}\|_{L^\infty}+1)^{k}\|u\|_{L^2(\mathbb
B)}.$$ Let $\sigma^{k}=\frac{1}{2}$,
$$
\|u\|_{L^\infty( \mathbb B_{\frac{1}{2}})}\leq C
(\|\overline{V}\|_{L^\infty}+1)^{C}\|u\|_{L^2(\mathbb B)},$$ where
$C$ depends on only $n$ and $m$. If $n=4m$, we will have the similar
result by applying the bootstrap arguments twice. By a rescaling
argument, we have
$$
\|u\|_{L^\infty( \mathbb B_{\delta})}\leq C
(\|\overline{V}\|_{L^\infty}+1)^{C}\delta^{-\frac{n}{2}}\|u\|_{L^2(\mathbb
B_{2\delta})},$$ if $0<\delta<\frac{1}{2}.$ Furthermore, we get
$$\|u\|_{L^\infty(\mathbb B_{r})}\leq
Cv^{C}(\rho-r)^{-\frac{n}{2}}\|u\|_{L^2(\mathbb B_{\rho})}$$ for
$0<r<\rho<\frac{1}{2}$. Recall that
$v=(\|\overline{V}\|_{L^\infty}+1)\leq 2M$. Thus,
$$\|u\|_{L^\infty(\mathbb B_{r_2})}\leq
CM^{C}(r_3-2r_2)^{-\frac{n}{2}}\|u\|_{L^2(\mathbb
B_{\frac{r_2+r_3}{3}})}.$$ Based on Lemma \ref{hhh} and the latter
inequality, we deduce that
\begin{equation}
\|u\|_{L^\infty(\mathbb B_{r_2})}\leq C
\exp{(CM)}(r_3-2r_2)^{-\frac{n}{2}}(\frac{3r_3}{2(r_2+r_3)})^{CM}h^{\frac{\alpha'}{\alpha'+\beta'}}(\frac{r_1}{2})
h^{\frac{\beta'}{\alpha'+\beta'}}(r_3), \label{end}
\end{equation}
where $$ \alpha'=\log\frac{3r_3}{2(r_2+r_3)}  $$ and
$$\beta'=C\log\frac{2(r_2+r_3)}{3r_1}.$$
Taking Lemma \ref{afa} and (\ref{rrr}) into account with $p=2$, we
have
$$
h({\frac{r_1}{2}})\leq C
(\|\overline{V}\|_{L^\infty}+1)r_1^{-2m}\|u\|_{L^2(\mathbb
B_{r_1})}$$ and
$$
h({{r_3}})\leq
C(\|\overline{V}\|_{L^\infty}+1)r_3^{-2m}\|u\|_{L^2(\mathbb
B_{2r_3})}.$$ It is true that
$$r^{-2m}\|u\|_{L^2(\mathbb B_{r})}\leq
r^{n/2-2m}\|u\|_{L^\infty(\mathbb B_{r})}\leq
\|u\|_{L^\infty(\mathbb B_{r})}$$ if $0<r<1$ and $n\geq 4m$. From
(\ref{end}) and the last three inequalities, we obtain
$$\|u\|_{L^\infty(\mathbb B_{r_2})}\leq
C\exp{(CM)}(r_3-2r_2)^{-\frac{n}{2}}(\frac{3r_3}{2(r_2+r_3)})^{CM}\|u\|^{\frac{\alpha'}
{\alpha'+\beta'}}_{L^\infty(\mathbb
B_{r_1})}\|u\|^{\frac{\beta'}{\alpha'+\beta'}}_{L^\infty(\mathbb
B_{2r_{3}})}.$$ By a rescaling argument, we arrive at the conclusion
of the lemma.
\end{proof}

We begin to prove Theorem \ref{th2}. The idea is  similar to the
proof of Theorem \ref{th1}. We also use the propagation of smallness
argument.
\begin{proof} [Proof of Theorem  \ref{th2}] We choose a small $r$ such that
$$\sup_{\mathbb B_{\frac{r}{2}(0)}}|u|=\epsilon,$$ where
$\epsilon>0$. Since $\sup_{|x|\leq 1}|u(x)|\geq 1$, there should
exist some $\bar x\in \mathbb B_1$ such that $u(\bar
x)=\sup_{|x|\leq 1}|u(x)|\geq 1$. We select a sequence of balls with
radius $r$ centered at $x_0=0, \ x_1, \cdots, x_d$ so that
$x_{i+1}\in \mathbb B_{\frac{r}{2}}(x_i)$ and $\bar x\in \mathbb
B_{r}(x_d)$, where $d$ depends on the radius $r$ which is to be
fixed. Employing the $L^\infty$-version of  three-ball lemma (i.e.
Lemma \ref{tir2}) with $r_1=\frac{r}{2}$, $r_2=r$, and $r_3=6r$ and
the boundedness assumption of $u$, we get
$$ \|u\|_{L^\infty{(\mathbb B_r)}}\leq C_1r^{-\frac{n}{2}}\epsilon^{\theta}\exp{(CM)} $$
where $1<\theta=\frac{\log 9/8}{\log 9/8+C\log 16/3}<1$ and $C_1$
depends on the $L^\infty$-norm of $u$, $n$ and $m$.

Iterating the above argument with $L^\infty$-version of three-ball
lemma for ball centered at $x_i$ and using the fact that
$\|u\|_{L^\infty(\mathbb B_\frac{r}{2}(x_{i+1}))}\leq
\|u\|_{L^\infty(\mathbb B_r(x_{i}))}$, we have
$$\|u\|_{L^\infty(\mathbb B_r(x_{i}))}\leq C_i
\epsilon^{D_i}r^{-\frac{F_in}{2}}\exp{(E_iM)}$$ for $i=0, 1, \cdots,
d$, where $C_i$ is a constant depending on $d$ and $L^\infty$-norm
of $u$, and $D_i$, $E_i$, $F_i$ are constants depending on $d$. By
the fact that $ u(\bar x)\geq 1$ and $\bar x \in \mathbb
B_{r}(x_d)$, we obtain
$$\sup_{\mathbb B_{\frac{r}{2}(0)}}|u|=\epsilon\geq
K_1\exp{(-K_2M)}r^{\frac{K_3n}{2}},$$ where $K_1$ is a constant
depending on $d$ and $L^\infty$-norm of $u$, and $K_2$, $K_3$ are
constants depending on $d$.

Applying Lemma \ref{tir2} again centered at origin with
$r_2=\frac{r}{2}>r_1$ and $r_3=3r$, where $r_1$ is sufficiently
small, we have
$$ K_1\exp{(-K_2M)}r^{\frac{K_3n}{2}}\leq C\exp{(CM)}r^{-\frac{n}{2}}
\|u\|_{L^\infty(\mathbb
B_{r_1})}^{\frac{\alpha_3}{\alpha_3+\beta_3}}
C_0^{\frac{\beta_3}{\alpha_3+\beta_3}}.
$$
Recall that $C_0$ is the $L^\infty$-norm for $u$ in $ \mathbb
B_{10}$. Then
\begin{equation}
K_4^{1+q}\exp{(-(1+q)K_5M)}r^{K_6(1+q)}\leq \|u\|_{L^\infty(\mathbb
B_{r_1})}, \label{fina3}
\end{equation}
 where $K_4$ depends on $d$ and $L^\infty$-norm of $u$, and $K_5$, $K_6$ depend on $d$,
$$ q=\frac{\beta_3}{\alpha_3}=\frac{C\log\frac{4r}{3r_1}}{\log\frac{9}{8}}= -K_7+C\log\frac{1}{r_1},$$
with constant $K_7>0$ depending on $r$. At this moment we fix the
small value of $r$. For instance, let $r=\frac{1}{100}$. Then the
value of $d$ is determined too. The inequality (\ref{fina3}) implies
that
$$\|u\|_{L^\infty(\mathbb B_{r_1})}\geq K_8r_1^{K_9M}, $$
where the constants $K_8, K_9$ depend on the dimension $n$, $m$, and
$C_0$. The proof of Theorem \ref{th2} is arrived.
\end{proof}

Thanks to Theorem \ref{th2}, we are able to prove the following
corollary for higher order elliptic equations in (\ref{bkh}), which
characterizes the asymptotic behavior of $u$ at infinity.
\begin{proof}[Proof of Corollary  \ref{cor2}]
We adapt the proof in \cite{K}. Since $u$ is continuous, we can find
$|x_0|=R$ so that $M(R)=\sup_{B_1(x_0)}|u(x)|$. Let
$$u_R(x)=u(R(x+\frac{x_0}{R})) \ \ \mbox{and}
\ \ \overline{V}_R=\overline{V}(R(x+\frac{x_0}{R})).$$ Then
$$(-\lap)^m u_R=R^{2m}\overline{V}_Ru_R,$$
with $\|u_R\|_{L^\infty}\leq C_0$ and
$\|R^{2m}\overline{V}_R\|_{L^\infty}\leq R^{2m}$. So $M=R^{2m}$ in
the notation of Theorem \ref{th2}. If $\bar x_0=\frac{-x_0}{R}$,
then $|\bar x_0|=1$ and $u_R(\bar x_0)=u(0)=1$. Hence $\|
u_R\|_{L^\infty (B_1)}\geq 1$. Note that $\sup_{\mathbb
B_1(x_0)}|u(x)|=\sup_{\mathbb B_{r_0}}|u_R(y)|$, where
$r_0=\frac{1}{R}$. The conclusion in Theorem \ref{th2} leads to
\begin{equation}
\begin{array}{lll}
M(R)=\sup_{\mathbb B_{r_0}}|u_R(y)|&\geq& C r_0^{CM} \nonumber\medskip\\
&=&C(\frac{1}{R})^{CR^{2m}} \nonumber \medskip\\
&=&C\exp(-CR^{2m}\log R),
\end{array}
\end{equation}
where $C$ depends on $n$, $m$ and $C_0$. Thus, the corollary
follows.
\end{proof}

\section{Strong unique continuation}
 In the rest of the paper, we will show the strong unique
continuation result for higher order elliptic equations by the
monotonicity of frequency function. This variant of frequency
function is also powerful in obtaining unique continuation results.
For strong unique continuation results of semilinear equations and
system of equations using frequency function, we refer to \cite{GL},
\cite{GL1} and \cite{AM} for the Lam\'{e} system of elasticity. Let
$u$ be the solution in (\ref{fff2}). Since we do not need to control
the vanishing order of solutions, we assume $\alpha=0$ for $H(r)$,
i.e.
$$ H(r)=\sum^m_{i=1}\int_{\mathbb B_r}u_i^2\,dx.$$
 We can check
that
\begin{equation}
H'(r)=\frac{n}{r}H(r)+\frac{1}{r}I(r), \label{comew}
\end{equation}
where
$$ I(r)=2\sum^m_{i=1}\int (x\cdot\nabla
u_i)u_i \,dx.$$ We consider the following frequency function
\begin{equation}
N(r)=\frac{I(r)}{H(r)}.
\end{equation}
Similar arguments as the proof of Lemma \ref{mono2} lead to
following monotonicity.
\begin{lemma}
There exists a constant $C$ depending only on $n, m$ such that
$$\exp{(Cr)}(N(r)+(\|\overline{V}\|_{L^\infty_{loc}}+1)^2) $$
is nondecreasing function of $r\in (0, 1)$. \label{mono3}
\end{lemma}
Based on the monotonicity property in above lemma, we are able to
show the proof of Theorem  \ref{th3}.

\begin{proof}[Proof of Theorem  \ref{th3}]
 By the equality
(\ref{comew}), we get
\begin{equation}(\log\frac{H(r)}{r^n})'=\frac{\overline{N}(r)\exp{(-Cr)}-v^2}{r},
\label{ajm}
\end{equation}
where
$$
\overline{N}(r)=\exp{(Cr)}(N(r)+v^2).$$ Recall that
$v=\|\overline{V}\|_{L^\infty_{loc}}+1$. Integrating from $R$ to
$4R$ for the equality (\ref{ajm}) yields that
$$\log{4^{-n}\frac{H(4R)}{H(R)}}\leq C\overline{N}(1)\log 4,$$
where $R$ is chosen to be small and $C$ depends on $n$ and $m$.
Taking exponential of both sides,
$$H(4R)\leq \exp{(C(\overline{N}(1)+1))}H(R),$$ that is,
\begin{equation}
\sum^{m}_{i=1}\int_{\mathbb B_{4R}} u_i^2\,dx\leq C
\sum^{m}_{i=1}\int_{\mathbb B_{R}}u_i^2\,dx, \label{fff}
\end{equation}
where $C$ depends on $\overline{N}(1) $. From the decomposition in
(\ref{syst}) and scaling arguments in (\ref{rrr}), we have
$$ (\sum^m_{i=1}\int_{\mathbb B_R}u_i^2\,dx)^{\frac{1}{2}}\leq CR^{-2m}
(\|\overline{V}\|_{L^\infty_{loc}}+1)\|u\|_{L^2(\mathbb B_{2R})},$$
 Therefore, with the aid of
(\ref{fff}),
$$\int_{\mathbb B_{4R}}u^2\,dx\leq \sum^m_{i=1}\int_{\mathbb
B_{4R}}u_i^2\,dx\leq C\sum^m_{i=1}\int_{\mathbb B_{R}}u_i^2\,dx\leq
C R^{-4m}\int_{\mathbb B_{2R}}u^2\,dx. $$ Thus, we get a doubling
type estimate
\begin{equation}
\int_{\mathbb B_{2R}}|u|^2\,dx \leq C R^{-4m}\int_{\mathbb
B_{R}}|u|^2\,dx, \label{dou}
\end{equation}
where $C$ depends on $\overline{N}(1)$,
$\|\overline{V}\|_{L^\infty_{loc}}$, $n$ and $m$. Now we fix $R$ and
prove that $u(x)\equiv 0$ on $\mathbb B_R$ from (\ref{dou}). The
argument is standard. See e.g. \cite{GL} on page 256-257.
\begin{equation}
\int_{\mathbb B_{R}}u^2\,dx\leq (C R^{-4m})^k\int_{\mathbb
B_{2^{-k}R}}u^2\,dx=(C R^{-4m})^k|\mathbb
B_{2^{-k}R}|^\beta\frac{1}{|\mathbb B_{2^{-k}R}|^\beta}\int_{\mathbb
B_{2^{-k}R}}u^2\,dx,\label{over}
\end{equation}
where the constant $\beta$ to be fixed. We choose $\beta$ such that
$CR^{-4m}2^{-n\beta}=1$. It yields that
\begin{equation}
\int_{\mathbb B_{R}}u^2\,dx\leq C R^{n\beta}\frac{1}{|\mathbb
B_{2^{-k}R}|^\beta}\int_{\mathbb B_{2^{-k}R}}u^2\,dx\to 0 \quad
\mbox{as}\ k\to \infty
\end{equation}
because of (\ref{vani}). Then $u\equiv 0$ in $\mathbb B_R$. Since we
can choose $\mathbb B_R$ arbitrarily in $\Omega$, the proof of
Theorem \ref{th3} follows.

\end{proof}

\section{Acknowledgement}

The author is  indebted to Professor C.D. Sogge for pointing out the
problem of nodal sets which leads to this work. The author is very
grateful to him for helpful discussions and encouragement. The
author also would like to thank Professor C.E. Kenig for very useful
comments on the earlier version of the manuscript.

\end{document}